%% file: MinkowskiSimplex_revised.tex
\documentclass[12pt]{amsart}

\usepackage{amsfonts,amssymb,amscd,amsmath,enumerate,verbatim,color}
\usepackage[latin1]{inputenc}
\usepackage{amscd}
\usepackage{latexsym}

\def\sort{{\rm sort}}

\def\ab{{\bf a}}
\def\bb{{\bf b}}
\def\eb{{\bf e}}
\def\tb{{\bf t}}
\def\ub{{\bf u}}
\def\vb{{\bf v}}
\def\xb{{\bf x}}
\def\yb{{\bf y}}
\def\zb{{\bf z}}
\def\wb{{\bf w}}

\def\Fc{{\mathcal F}}
\def\Pc{{\mathcal P}}
\def\Qc{{\mathcal Q}}
\def\Bc{{\mathcal B}}
\def\Ac{{\mathcal A}}

\def\RR{{\mathbb R}}
\def\ZZ{{\mathbb Z}}
\def\QQ{{\mathbb Q}}

\newtheorem{Theorem}{Theorem}[section]

\newtheorem{Corollary}[Theorem]{Corollary}
\newtheorem{Proposition}[Theorem]{Proposition}

\theoremstyle{definition}

\newtheorem{Example}[Theorem]{Example}

\newtheorem{Conjecture}[Theorem]{Conjecture}

\begin{document}

\title{Toric ideals of Minkowski sums of unit simplices}
\author{Akihiro Higashitani and Hidefumi Ohsugi}

\address{Akihiro Higashitani,
Department of Pure and Applied Mathematics,
Graduate School of Information Science and Technology,
Osaka University,
Suita, Osaka 565-0871, Japan}
\email{higashitani@ist.osaka-u.ac.jp}

\address{Hidefumi Ohsugi,
Department of Mathematical Sciences,
School of Science and Technology,
Kwansei Gakuin University,
Sanda, Hyogo 669-1337, Japan} 
\email{ohsugi@kwansei.ac.jp}


\begin{abstract}
In this paper, we discuss the toric ideals of the Minkowski sums of unit simplices. 
More precisely, we prove that the toric ideal of the Minkowski sum of unit simplices has a squarefree initial ideal and is generated by quadratic binomials. 
Moreover, we also prove that the Minkowski sums of unit simplices have the integer decomposition property. 
Those results are a partial contribution to Oda conjecture and B{\o}gvad conjecture. 
\end{abstract}

\maketitle

\section{Introduction}

\subsection{Terminologies}

Let $P \subset \RR^d$ be a {\em lattice polytope}, which is a convex polytope all of whose vertices belong to the standard lattice $\ZZ^d$, of dimension $d$. 
We say that $P$ has the {\em integer decomposition property} (or is {\em IDP} for short) if for any positive integer $n$ and $\alpha \in n P \cap \ZZ^d$, 
there exist $\alpha_1,\ldots,\alpha_n \in P \cap \ZZ^d$ such that $\alpha=\alpha_1+ \cdots + \alpha_n$. 

We say that $P$ is {\em smooth} if for each vertex $v$ of $P$, 
the set of the primitive edge direction vectors of $v$ forms a $\ZZ$-basis for $\ZZ^d$. 
A smooth polytope is often said to be a {\em Delzant polytope}. 
It is known that $P$ is smooth if and only if the projective toric variety associated to $P$ is smooth. 

Let $\Ac=\{\ab_1,\ldots,\ab_m\} \subset \ZZ^d$. 
We allow $\ab_i = \ab_j$ for some $1\le i < j\le m$, that is,
$\Ac$ may be a multiset.
We say that $\Ac$ is a {\em configuration} if there is a hyperplane $H \subset \RR^d$ 
which is of the form $c_1x_1+\cdots+c_dx_d = 1$ with $c_i \in \QQ$ such that $\Ac \subset H$. 
Assume that $\Ac$ is a configuration. Let $K[\tb^\pm]=K[t_1^\pm,\ldots,t_d^\pm]$ be the Laurent polynomial ring in $d$ variables over a field $K$. 
Given a lattice point $\ab=(a_1,\ldots,a_d) \in \ZZ^d$, we set $\tb^\ab=t_1^{a_1} \cdots t_d^{a_d} \in K[\tb^\pm]$. 
The {\em toric ring} of $\Ac$ is the subalgebra $K[\Ac]$ of $K[\tb^\pm]$ generated by $\tb^{\ab_1},\ldots,\tb^{\ab_m}$. 
The {\em toric ideal} $I_\Ac$ of $\Ac$ is the defining ideal of the toric ring $K[\Ac]$, i.e., it is 
the kernel of a surjective ring homomorphism $\pi:K[x_1,\ldots,x_m] \rightarrow K[\Ac]$ defined by $\pi(x_i)=\tb^{\ab_i}$. 
It is known that $I_\Ac$ is generated by homogeneous binomials. 
Given a lattice polytope $P \subset \RR^d$, the toric ideal of $P$ stands for the toric ideal of the configuration $\Ac_P \subset \ZZ^{d+1}$, 
where $\Ac_P=\{(\alpha,1) \in \ZZ^{d+1} : \alpha \in P \cap \ZZ^d\}$. 
We refer the readers to \cite{binomialideals, Sturmfels} for the introduction to toric ideals and their Gr\"obner bases. 

\subsection{Two conjectures on smooth polytopes}

Traditionally, the theories of lattice polytopes and toric geometry have been developing by interacting each other. 
In particular, the following two conjectures are of great importance from viewpoints of not only combinatorics but toric geometry: 

\begin{Conjecture}[Oda Conjecture]\label{Oda}
Every smooth polytope is IDP. 
\end{Conjecture}
\begin{Conjecture}[B{\o}gvad Conjecture]\label{Bogvad}
The toric ideal of every smooth polytope is generated by quadratic binomials. 
\end{Conjecture}
Note that there is no logical connection between these two conjectures. 
In fact, there are many examples of smooth polytopes which are IDP but their toric ideals are not generated by quadratic binomials, and vice versa. 
There are some partial results on these two conjectures. See \cite{Bruns}. 
Those conjectures are true in dimension $2$. 
Recently, it was proved in \cite{BHHHJKM} that Conjecture~\ref{Oda} is true for centrally symmetric smooth polytopes of dimension $3$, 
but it is still open in dimension (at least) $3$ in general.

\subsection{Generalized Permutohedra}

We recall {\em generalized permutohedra}, introduced by Postnikov \cite[Section 6]{Postnikov}. 
Given a positive integer $n$, let $[n]=\{1,2,\ldots,n\}$. We define a convex polytope $P_n^Z(\{z_I\})$ as follows: 
$$P_n^Z(\{z_I\})=\left\{(t_1,\ldots,t_n) \in \RR^n : \sum_{i=1}^n t_i= z_{[n]}, \sum_{i \in I} t_i \geq z_I \text{ for each } I \subsetneq [n]\right\},$$
where $\{z_I\}$ is a given collection of parameters with $z_I \geq 0$ for each nonempty $I \subset [n]$ and belongs to a certain full-dimensional polyhedral subset of $\RR^{2^n-1}$. 
Note that $P_n^Z(\{z_I\})$ is a usual permutohedron if $z_I=z_J$ whenever $|I|=|J|$. 

On the other hand, let $$P_n^Y(\{y_I\})=\sum_{I \subset [n]}y_I \Delta_I,$$
where $\{y_I\}$ is a given collection of parameters with $y_I \geq 0$ for each nonempty $I \subset [n]$, 
and $\sum$ stands for the Minkowski sum of polytopes. It is proved in \cite[Proposition 6.3]{Postnikov} that 
for a given $\{y_I\}$, 
we see that $P_n^Y(\{y_I\})=P_n^Z(\{z_I\})$
by setting $z_I=\sum_{J \subset I}y_J$. 
Thus $\{P_n^Y(\{y_I\}) : y_I \ge 0 \mbox{ for } I \subset [n]\}$ is a special (but large enough) class of generalized permutohedra. 
Moreover, {\em nestohedra}, which are a wide class of smooth polytopes, can be described as $P_n^Y(\{y_I\})$ with $y_I \ge 0$ for $I \subset [n]$ 
and those include many other important classes of smooth polytopes. See Section~\ref{nestohedron} for more details. 
The main object of this paper is $P_n^Y(\{y_I\})$ in the case $y_I \in \ZZ_{\geq 0}$.

\subsection{Results}

The main theorem of the present paper is the following: 
\begin{Theorem}
\label{main}
Let $y_I \in \ZZ_{\geq 0}$ for each $I \subset [n]$. Then the generalized permutohedron $P_n^Y(\{y_I\})$ satisfies the following{\rm :} 
\begin{itemize}
\item[(a)] 
$P_n^Y(\{y_I\})$ is IDP{\rm ;} 
\item[(b)]
The toric ideal of $P_n^Y(\{y_I\})$ has a squarefree initial ideal{\rm ;}
\item[(c)] The toric ideal of $P_n^Y(\{y_I\})$ is generated by quadratic binomials. 
\end{itemize}
\end{Theorem}

Since nestohedra, which can be written as $P_n^Y(\{y_I\})$ above, are smooth polytopes, we immediately obtain the following corollary 
that is a partial contribution to Oda and B{\o}gvad Conjectures: 

\begin{Corollary}
Oda Conjecture and B{\o}gvad Conjecture are true for nestohedra. 
\end{Corollary}

\subsection{Organization}
The present paper is organized as follows. In Section~\ref{nestohedron}, we recall the notion of nestohedra and graph associahedra. 
Note that graph associahedra are a subclass of nestohedra. 
In Section 3, we review the key notion of generalized nested configurations, introduced by Shibuta \cite{Shibuta}. 
Finally, in Section 4, we prove Theorem~\ref{main} by using the notion of
generalized nested configurations.

\subsection*{Acknowledgements}
The authors are grateful to anonymous referees
for their careful reading of the manuscript and for their comments. 
The authors are partially supported by JSPS KAKENHI $\sharp$17K14177 and $\sharp$18H01134.

\bigskip

\input{nestohedron_revised}

\bigskip

\input{quadratic_revised}

\bigskip

\input{references_revised}
\end{document}

%% file: nestohedron_revised.tex
\section{Nestohedra and Graph associahedra}\label{nestohedron}

In this section, we recall the notion of nestohedra and graph associahedra. 
As explained in Introduction, nestohedra are a kind of generalized permutohedra, and graph associahedra are a kind of nestohedra. 
See \cite[Section 1.5]{BuchstaberPanov} for the introduction to nestohedra and graph associahedra.

Given a subset $S \subset [n]$, let $\Delta_S$ denote the convex hull of $\{\eb_i : i \in S \} \subset \RR^n$, 
where $\eb_i$ is the $i$-th unit vector of $\RR^n$. 
For a collection $\Fc$ of subsets of $[n]$, we set $$\Pc_\Fc := \sum_{S \in \Fc}\Delta_S.$$
Clearly, $\Pc_\Fc$ is a lattice polytope.

Let $\Bc$ be a collection of subsets of $[n]$. We say that $\Bc$ is a {\em building set} if $\Bc$ satisfies 
\begin{itemize}
\item[(i)] if $I,J \in \Bc$ with $I \cap J \neq \emptyset$, then $I \cup J \in \Bc$; and 
\item[(ii)] $\{i\} \in \Bc$ for each $i=1,\ldots,n$. 
\end{itemize}
Note that the condition (ii) is just added only for convenience. 
The lattice polytope $\Pc_\Bc$ associated to a building set $\Bc$ is called a {\em nestohedron}. 
Since taking the Minkowski sum of some polytope and $\Delta_{\{i\}}$ is nothing but a parallel transformation, 
we may treat $\Bc \setminus \{\{1\},\ldots,\{n\}\}$ instead of $\Bc$. 

We consider the building set arising from a finite simple graph. 
Let $G$ be a finite simple graph on the vertex set $[n]$ with the edge set $E(G)$. 
Let $\Bc_G$ be a collection of all subsets $S$ of $[n]$ such that the induced subgraph of $G$ on $S$ is connected. 
Then it is easy to see that $\Bc_G$ is a building set, called a {\em graphical building set}. 
The nestohedron associated to a graphical building set is called a {\em graph associahedron}. 
Graph associahedra include the following important classes of smooth polytopes (see \cite[Section 8]{Postnikov}): 
\begin{itemize}
\item Let $K_n$ be a complete graph on $[n]$. Then the graph associahedron of $K_n$ is a (kind of) permutohedron. 
\item Let $G$ be a path graph on $[n]$, i.e., the graph whose edge set is $\{\{i,i+1\} : i =1,\ldots,n-1\}$. 
Then its graph associahedron is the {\em associahedron}, also known as the {\em Stasheff polytope}. 
\item Let $G$ be a cycle graph on $[n]$, i.e., the graph whose edge set is $\{\{i,i+1\} : i =1,\ldots,n-1\} \cup \{\{1,n\}\}$. 
Then its graph associahedron is the {\em cyclohedron}, also known as the {\em Bott--Taubes polytope}. 
\item Let $\Bc=\{ [i] : i=2,3,\ldots,n\} \cup \{\{1\},\{2\},\ldots,\{n\}\}$. Then $\Bc$ is a building set, but not a graphical one. 
The nestohedron $\Pc_\Bc$ is exactly the polytope studied by Pitman and Stanley \cite{PitmanStanley}, and  called the {\em Pitman--Stanley polytope}. 
\end{itemize}

It is proved in \cite[Proposition 7.10]{Postnikov} that the generalized permutohedron $P_n^Y(\{y_I\})$ is smooth 
if $\{I \subset [n] : y_I >0\}$ is a building set. 
Remark that this is not necessary, i.e., there are examples of smooth generalized permutohedra not associated with building sets.

%% file: quadratic_revised.tex
\section{Generalized nested configurations}

In the present section, we explain the notion of 
generalized nested configurations introduced by 
Shibuta \cite[Section~3.3]{Shibuta} as an application of the results on Gr\"obner bases of contraction ideals.

Let $A \subset \ZZ_{\ge 0}^s$
and $B_i = \{\bb_1^{(i)},\dots,\bb_{\lambda_i}^{(i)} \}
\subset \ZZ^n$ ($i = 1,2,\dots, s$) be configurations.
Then the {\em generalized nested configuration}
arising from $A$ and $B_1,\dots,B_s$ is the configuration 
$A[B_1,\dots,B_s]$ in $\ZZ^n$ defined by
$$
A[B_1,\dots,B_s]:=
\left\{
\sum_{i=1}^s \sum_{j=1}^{\lambda_i} a_j^{(i)} \bb_j^{(i)}
:
a_j^{(i)} \in \ZZ_{\ge 0}, 
\left( \sum_{j=1}^{\lambda_1} a_j^{(1)},
\dots, 
\sum_{j=1}^{\lambda_s} a_j^{(s)}
   \right)
\in A
\right\}
.$$
This is a generalization of nested configurations introduced in \cite{AHOT}.

\begin{Example}
\label{exa1}
Let $A = \{(1,\dots,1)\} \subset \ZZ_{\ge 0}^s$
and $B_1,\dots,B_s \subset \ZZ^n$ be configurations.
Then $A[B_1,\dots,B_s] = \{ \bb_1 + \dots + \bb_s : \bb_i \in B_i  \}.$
\end{Example}

Shibuta \cite[Theorem~3.5]{Shibuta} proved the following. 

\begin{Proposition}
\label{ShibutaTheorem}
Let $K[\zb^{\pm 1}]=K[z_1^{\pm 1},\dots,z_n^{\pm 1}]$
be a Laurent polynomial ring over a field $K$ with $\deg (z_i) = \vb_i \in \QQ^d$
and let $\ub_1,\dots,\ub_s \in \QQ^d$ be rational vectors that are linearly independent over $\QQ$. 
Suppose that configurations $B_1,\dots,B_s$ in $\ZZ^n$
satisfies $B_i \subset \{ \bb \in \ZZ^n : \deg (\zb^\bb) = \ub_i\}$ for $i = 1,2,\dots,s$.
Let $A \subset \ZZ_{\ge 0}^s$ be a configuration
and let $B = B_1 \cup \dots \cup B_s$.
Then we have the following{\rm :}

\begin{itemize}
\item[{\rm (a)}]
If both $I_A$ and $I_B$ possess initial ideals of 
degree at most $m$, 
then so does $I_{A[B_1,\dots,B_s]}$.

\item[{\rm (b)}]
If both $I_A$ and $I_B$ possess squarefree initial ideals,
then so does $I_{A[B_1,\dots,B_s]}$.
\end{itemize}
\end{Proposition}

Next, we explain how to construct a corresponding Gr\"obner basis of $I_{A[B_1,\dots,B_s]}$ in Proposition~\ref{ShibutaTheorem}.
Work with the same assumption as in Proposition~\ref{ShibutaTheorem}.
Let ${\mathcal E}_i = \{\eb_1^{(i)},\dots,\eb_{\lambda_i}^{(i)}\} \subset \bigoplus_{i=1}^s \bigoplus_{j=1}^{\lambda_i}\ZZ \eb_j^{(i)} $
for $i = 1,2,\dots,s$.
We define polynomial rings $K[\xb]$ and $K[\yb] $
over a field $K$ by 
\begin{eqnarray*}
K[\xb] &=& K[x_\ab : \ab \in A[{\mathcal E}_1,\dots,{\mathcal E}_s] ],\\
K[\yb] &=& K[y_j^{(i)} : i \in [s], j \in [\lambda_i]].
\end{eqnarray*}
Then the toric ideal $I_{A[{\mathcal E}_1,\dots,{\mathcal E}_s]}$ of $A[{\mathcal E}_1,\dots,{\mathcal E}_s]$ is the kernel of ring homomorphism
$$
\varphi_{A[{\mathcal E}_1,\dots,{\mathcal E}_s]}: K[\xb] \rightarrow K[\yb] , \ \ \varphi_{A[{\mathcal E}_1,\dots,{\mathcal E}_s]} (x_\ab) = \yb^\ab,
$$
and the toric ideal $I_B$ of $B$ is the kernel of ring homomorphism
$$
\varphi_B : K[\yb] \rightarrow K[\zb^{\pm 1}] , \ \ \varphi_B (y_j^{(i)}) = \zb^{\bb_j^{(i)}}.
$$
Then we have
$$
I_{A[B_1,\dots,B_s]}=
\ker (\varphi_B \circ \varphi_{A[{\mathcal E}_1,\dots,{\mathcal E}_s]})
=
\varphi_{A[{\mathcal E}_1,\dots,{\mathcal E}_s]}^{-1}(
\ker (\varphi_B))
=\varphi_{A[{\mathcal E}_1,\dots,{\mathcal E}_s]}^{-1}( I_B).
$$
Given an element $q$ of the toric ring 
$
K[A[{\mathcal E}_1,\dots,{\mathcal E}_s]] 
(= 
{\rm im} (\varphi_{A[{\mathcal E}_1,\dots,{\mathcal E}_s]}) )$, there exists a unique polynomial $\tilde{q} \in K[\xb]$ such that $\varphi_{A[{\mathcal E}_1,\dots,{\mathcal E}_s]} (\tilde{q}) = q$ and 
any monomial of $\tilde{q}$ does not belong to 
the initial ideal of $I_{A[{\mathcal E}_1,\dots,{\mathcal E}_s]}$.
Then we define ${\rm lift} (q) = \tilde{q}$.
Since $\ub_1,\dots,\ub_s$ are linearly independent over $\QQ$,
it follows that $I_B$ is homogeneous with respect to a multi-grading $\deg(y_j^{(i)}) = \eb_i \in \ZZ^s$.
With respect to this grading, let $K[\yb] = \bigoplus_{\ub \in \ZZ^s} K[\yb]_{\ub}$, 
where $K[\yb]_{\ub}$ is the $K$-vector space spanned by all monomials in $K[\yb]$ of multi-degree $\ub$.
A corresponding Gr\"obner basis can be constructed by the following way: 

\begin{Proposition}[{\cite[Proposition~2.28]{Shibuta}}]\label{contractionGB}
Work with the same notation and assumptions as above.
Let $G$ be the reduced Gr\"obner basis of $I_{A[{\mathcal E}_1,\dots,{\mathcal E}_s]}$,
and let $F = \{f_1,\dots,f_\ell\}$ be that of $I_B$ with $\deg (f_i) = \vb_i$.
Then
$$
G \cup \{ {\rm lift} (\yb^\ab \cdot f_i) : i \in [\ell] , \ \yb^\ab \in \Gamma (\vb_i)\}
$$
is a Gr\"obner basis of $I_{A[B_1,\dots,B_s]}$,
where each $\Gamma (\vb_i)$ is the minimal set of monomial generators of 
$K[A[{\mathcal E}_1,\dots,{\mathcal E}_s]]$-submodule
$$
\bigoplus_{ \yb^{\ub + \vb_i} \in K[A[{\mathcal E}_1,\dots,{\mathcal E}_s]]} K[\yb]_{\ub}
$$
of $K[\yb]$. 
\end{Proposition}

\section{Proof of Theorem \ref{main}}

Recall that the main object of this paper is
\begin{eqnarray}
\label{gp}
P_n^Y(\{y_I\})=\sum_{I \subset [n]}y_I \Delta_I,
\end{eqnarray}
where $\{y_I\}$ is a given collection of parameters with $y_I \in \ZZ_{\geq 0}$. 
Since each $y_I$ is a nonnegative integer, the equation~(\ref{gp}) can be rewritten as
$$
P_n^Y(\{y_I\})=\sum_{I \subset [n]} (
 \underbrace{\Delta_I + \cdots + \Delta_I}_{y_I}).
$$
Hence $P_n^Y(\{y_I\})$ coincides with 
\begin{eqnarray}
\label{tuplesum}
\Pc_\Fc = \Delta_{S_1} + \cdots + \Delta_{S_m},
\end{eqnarray}
where $\Fc = (S_1, \dots ,S_m)$ is a tuple of nonempty subsets $S_i \subset [n]$ such that each $I \subset [n]$
appears $y_I$ times in $\Fc$.
In order to study the Minkowski sum (\ref{tuplesum}), 
we consider the {\em Cayley sum} 
$$
\Qc_\Fc := {\rm conv} (
(\Delta_{S_1} \times \eb_1),  \dots , (\Delta_{S_m} \times \eb_m)
)
\subset \RR^{m+n}
$$
of $\Delta_{S_1}, \dots, \Delta_{S_m}$.
Let $G$ be a bipartite graph on the vertex set $\{1,2,\dots,n\} \cup \{1',2',\dots,m'\}$
whose edge set is
$
\{
\{j, i'\} : j \in [n], i \in [m] , j \in S_i
\}
.$
Then $\Qc_\Fc$ coincides with the {\em edge polytope} \cite[Section 5.2]{binomialideals} of $G$.
Here the edge polytope $\Pc(G)$ of a graph $G$ on the vertex set $[d]$ is the convex hull of 
$$
\{ \eb _i + \eb_j \in \RR^d : \{i,j\} \mbox{ is an edge of } G \}.
$$
Postnikov \cite{Postnikov} calls $\Pc(G)$ a {\em root polytope} of $G$.
The following is known. 

\begin{Proposition}[{\cite[Theorem~5.24]{binomialideals}}]\label{bipartite_unimodular}
Let $G$ be a bipartite graph.
Then the edge polytope $\Pc(G)$ of $G$ is unimodular
(every triangulation of $\Pc (G)$ is unimodular) and IDP.
In particular, $\Qc_\Fc$ is unimodular and IDP for any $\Fc$.
\end{Proposition}

It is known by \cite[Proposition 14.12]{Postnikov} that
$(\Delta_{S_1} \cap \ZZ^n) + \dots + (\Delta_{S_m}\cap \ZZ^n)=\Pc_\Fc \cap \ZZ^n $.
Moreover, the following is also known. 
\begin{Proposition}[{\cite[Theorem~0.4]{TCayley}}]\label{TsuchiyaCayley}
Let $P_1,\ldots,P_m \subset \RR^n$ be lattice polytopes. 
If the Cayley sum of $P_1,\ldots,P_m$ is IDP, then the Minkowski sum $\sum_{i=1}^m a_iP_i$ is IDP for any $a_1,\ldots,a_m \in \ZZ_{\geq 0}$. 
\end{Proposition}

We are now in a position to prove Theorem~\ref{main}.

\begin{proof}[Proof of Theorem~\ref{main}]
(a) Propositions~\ref{bipartite_unimodular} and \ref{TsuchiyaCayley} imply that $\Pc_\Fc$ is IDP.
%

(b)
We show that the toric ideal of
$\Pc_\Fc$ has a squarefree initial ideal by
 using Shibuta's theory of contraction ideals \cite{Shibuta}.
Let 
\begin{eqnarray*}
K[\xb] &=& K[x_{j_1, \dots, j_m} : j_k \in S_k \  ( 1 \le k \le m )],\\
K[\yb] &=& K[y_j^{(i)} : i \in [m], j \in S_i],\\
K[\zb, \wb] &=& K[z_1, \dots, z_n, w_1, \dots, w_m]
\end{eqnarray*}
be the polynomial rings over a field $K$.
We now consider the ring homomorphisms
$$
\varphi_A : K[\xb] \rightarrow K[\yb] , \ \ \varphi_A (x_{j_1, \dots, j_m}) = y_{j_1}^{(1)} \dots y_{j_m}^{(m)},
$$
$$
\varphi_B : K[\yb] \rightarrow K[\zb, \wb] , \ \ \varphi_B (y_j^{(i)}) =z_j w_i.
$$
Then we have the following:
\begin{itemize}
\item
The kernel of $\varphi_A$ is the toric ideal of the Segre product of the polynomial rings $ K[y_j^{(1)} : j \in S_1], \dots, K[y_j^{(m)} : j \in S_m]$. 
It is known that $\ker (\varphi_A)$ has a squarefree quadratic initial ideal with respect to 
a ``sorting order'', see \cite[Section~9.5]{binomialideals}.
Let $G$ be the corresponding quadratic Gr\"obner basis of $\ker (\varphi_A)$.

\item
The kernel of $\varphi_B$ is the toric ideal of $\Qc_\Fc$.
By \cite[Theorem~4.17]{binomialideals}, the initial ideal of 
$\ker (\varphi_B)$ is generated by squarefree monomials with respect to any monomial order
since $\Qc_\Fc$ is unimodular (Proposition~\ref{bipartite_unimodular}). 
Let $\{f_1, \dots, f_\ell\}$ be a Gr\"obner basis
of $I_{\Qc_\Fc}$.

\item
The composition $\varphi_B \circ \varphi_A : K[\xb] \rightarrow K[\zb, \wb]$ satisfies
$$
\varphi_B \circ \varphi_A (x_{j_1, \dots, j_m}) =
\varphi_B ( y_{j_1}^{(1)} \dots y_{j_m}^{(m)})
= z_{j_1} \dots z_{j_m} w_1 \dots w_m.
$$
Since the monomial $ w_1 \dots w_m$ appears in $\varphi_B \circ \varphi_A (x_{j_1, \dots, j_m}) $ for all variables $x_{j_1, \dots, j_m}$,
it follows that
the kernel of $\varphi_B \circ \varphi_A$ is equal to the kernel of
a homomorphism $\psi : K[\xb] \rightarrow K[\zb]$, where 
$
\psi (x_{j_1, \dots, j_m}) = z_{j_1} \dots z_{j_m}.
$
Then the kernel of $\psi$ is the toric ideal of 
the configuration $\Ac = \{ \sum_{k=1}^m \eb_{j_k} :    j_k \in S_k \  ( 1 \le k \le m )  \}$.
Here we must regard $\Ac$ as a multiset.
On the other hand, $\Ac$ coincides with the vertex set of $\Pc_\Fc$
{\em as a set}.
Thus the kernel of 
$\psi$ is equal to $I_{\Pc_\Fc} +J$, where $I_{\Pc_\Fc}$ is the toric ideal of $\Pc_\Fc$
and $J$ is generated by the linear forms $x_{j_1, \dots, j_m} - x_{k_1, \dots, k_m}$ such that  $\psi(x_{j_1, \dots, j_m}) = \psi( x_{k_1, \dots, k_m})$.
Note that $\psi(x_{j_1, \dots, j_m}) = \psi( x_{k_1, \dots, k_m})$
if and only if
$\sort(j_1 \cdots j_m) = \sort( k_1 \cdots k_m)$.
\end{itemize}

If we set $\deg z_j = {\bf 0}\in \QQ^m$, $\deg w_i = \eb_i \in \QQ^m$, 
and $B_i = \{(\ab, \eb_i) \in \ZZ^{m+n}: \ab \in \Delta_{S_i} \cap \ZZ^n \}$,
then $\eb_1,\dots,\eb_m$ are linearly independent over $\QQ$ 
and the assumptions in Proposition~\ref{ShibutaTheorem} are satisfied. 
Note that $\Qc_\Fc$ is the convex hull of $B = B_1 \cup \cdots \cup B_m$.
Since both $\ker (\varphi_A)$ and $\ker (\varphi_B)$ have squarefree initial ideals,
so does $\ker (\varphi_B \circ \varphi_A)$ by Proposition~\ref{ShibutaTheorem}.
Let ${\mathcal G}$ be the corresponding reduced Gr\"obner basis of $\ker (\varphi_B \circ \varphi_A)$.
By Proposition~\ref{contractionGB}, we have 
${\mathcal G} = G \cup \{ {\rm lift} (\yb^\ab \cdot f_i) : i \in [\ell] , \ \yb^\ab \in \Gamma (\vb_i)\}$.
Then ${\mathcal G} \setminus J$ is a Gr\"obner basis of $I_{\Pc_\Fc} $.
Thus the initial ideal of $I_{\Pc_\Fc} $ is squarefree. 

(c)
Since $G$ consists of quadratic binomials, it is enough to show that 
each ${\rm lift} (\yb^\ab \cdot f_i)$  is generated by the binomials in $ \ker (\varphi_B \circ \varphi_A)$ of degree $\le 2$.
%
It is known by \cite[Corollary~5.12]{binomialideals} that each $f_k$
corresponds to an even cycle in the bipartite graph.
In fact, if $\deg (f_k) =r $, then there exists a cycle $C = (q_1, p_1',q_2, p_2', \dots, q_r, p_r')$ of length $2r$ in the bipartite graph such that
$$
f_k = 
y_{q_1}^{(p_1)} y_{q_2}^{(p_2)} \dots y_{q_r}^{(p_r)} 
-
y_{q_2}^{(p_1)} y_{q_3}^{(p_2)} \dots y_{q_r}^{(p_{r-1})} y_{q_1}^{(p_r)} .
$$
By changing indices if needed, we may assume that
$C = (1,1',2,2',\dots, r, r')$
and 
$$
f_k = 
y_1^{(1)} y_{2}^{(2)} \dots y_r^{(r)} 
-
y_2^{(1)} y_{3}^{(2)} \dots y_r^{(r-1)}  y_1^{(r)}.
$$
Then ${\rm lift} (\yb^\ab \cdot f_k)$
is of the form
$$
{\rm lift} (\yb^\ab \cdot f_k)
=
x_{j_1^{1}, \dots, j_m^{1}}
x_{j_1^{2}, \dots, j_m^{2}}
\dots
x_{j_1^{s}, \dots, j_m^{s}}
-
x_{k_1^{1}, \dots, k_m^{1}}
x_{k_1^{2}, \dots, k_m^{2}}
\dots
x_{k_1^{s}, \dots, k_m^{s}},
$$
where $1 \le s \le r$.
Suppose that $s \ge 3$.
Since $\varphi_A ({\rm lift} (\yb^\ab \cdot f_k)) = \yb^\ab \cdot f_k$,
we have
\begin{equation}
\label{soeji}
\{k_\alpha^{1}, \dots, k_\alpha^{s}\}
=
\begin{cases}
\{j_\alpha^{1}, \dots, j_\alpha^{s}, \alpha+1\} 
\setminus \{\alpha \}
 & \mbox{ if } \alpha = 1, 2, \dots, r-1,\\
\\
\{j_\alpha^{1}, \dots, j_\alpha^{s}, 1\} 
\setminus \{\alpha \}
& \mbox{ if } \alpha =r, \\
\\
\{j_\alpha^{1}, \dots, j_\alpha^{s}\} & \mbox{ if } \alpha = r+1, r+2, \dots,  m
\end{cases}
\end{equation}
as multisets.
Note that $\ker (\varphi_B \circ \varphi_A)$
possesses the quadratic binomials of the form 
\begin{equation}
\label{2ji}
x_{j_1 , \dots , j_\xi ,\dots, j_m} 
x_{j_1' , \dots , j_\xi' ,\dots, j_m'} 
-
x_{j_1 , \dots , j_\xi' ,\dots, j_m} 
x_{j_1' , \dots , j_\xi ,\dots, j_m'}
\in \ker (\varphi_A).
\end{equation}
For example, we have
\begin{eqnarray*}
{\rm lift} (\yb^\ab \cdot f_k)
&=&
x_{j_1^{1},  \dots,  j_m^{1}}
x_{j_1^{2},  \dots,  j_m^{2}}
\dots
x_{j_1^{s},  \dots,  j_m^{s}}
-
x_{k_1^{1},  \dots,  k_{m-1}^{1}, k_m^{2}}
x_{k_1^{2},  \dots,  k_{m-1}^{2}, k_m^{1}}
\dots
x_{k_1^{s},  \dots,  k_m^{s}}\\
&&+ \ 
x_{k_1^{3},   \dots,  k_m^{3}}
\dots
x_{k_1^{s},  \dots,  k_m^{s}}
\left(
x_{k_1^{1},  \dots,  k_{m-1}^{1}, k_m^{2}}
x_{k_1^{2},  \dots,  k_{m-1}^{2}, k_m^{1}}
-
x_{k_1^{1},  \dots,  k_m^{1}}
x_{k_1^{2},  \dots,  k_m^{2}}
\right).
\end{eqnarray*}
By repeating this procedure, the equation (\ref{soeji}) guarantees that 
  ${\rm lift} (\yb^\ab \cdot f_k)$ is generated by quadratic binomials of the form (\ref{2ji})
and the binomial
$$
g=
\xb^\ub
(
x_{1,  2,  \dots,  r , j_{r+1}, \dots, j_m}
-
x_{2, \dots, r, 1, j_{r+1}, \dots, j_m}
).
$$
Then $x_{1,  2,  \dots,  r , j_{r+1}, \dots, j_m}
-
x_{2, \dots, r, 1, j_{r+1}, \dots, j_m}$ belongs to $J$.
\end{proof}